\documentclass{article}
\usepackage{amsmath, amsthm} 
\usepackage{amssymb} % blackboard 
\usepackage{bm} % bold greek 
\usepackage{graphicx}
\usepackage[metapost]{mfpic}
\usepackage{url}
\usepackage{bigstrut}
\usepackage{tikz} % https://www.sharelatex.com/learn/TikZ_package
\usetikzlibrary{matrix}

\newtheorem{Thm}{Theorem}

\newtheorem{Lem}[Thm]{Lemma} 
\newtheorem{Def}[Thm]{Definition} 

\long\def \delete #1{} % digests \par , not \section nor BibTeX !   
\def \sig {{\rm \,sign}} 
\def \dis {{\rm \,ht}} 

\begin{document}
\title{Sodalite Network: Height and Spherical Content (Coordination Sequence)} 
\author{
{W.~Fred Lunnon, NUI Maynooth;} 
}
\maketitle

\begin{abstract} 
The {\sl sodalite} network is the edge-skeleton of the uniform tiling in 
Euclidean {3-dimensional} space by Archimedean tetrakaidecahedra (truncated 
octahedra). We develop explicit expressions for its {\sl height} (minimum 
network path length from some fixed to given vertex) and {\sl coordination} 
(content of network sphere of given height) functions. The final 
discussion should to some extent assist in motivating and signposting our 
proof strategy, in the course of ruminating on its potential generalisation. 
\par 

\end{abstract} 

\par 

\section{Geometry and Symmetry} \label{[geomsymm]} 
\par 

The {\sl sodalite} network is the edge-skeleton of the uniform tiling in 
Euclidean {3-dimensional} space by Archimedean tetrakaidecahedra (truncated 
octahedra). It may be constructed by `bitruncating' vertices of the standard 
cubical tiling, so that octahedra forming around cubical vertices collide and 
in turn become truncated. 
\par 

After scaling up by a factor 4, the canonical cell centred (around a cubical 
vertex) at origin $(0,0,0)$ has 24 vertices 
\begin{eqnarray} \label{canoncell} 
  \nonumber & \{ (0,1,2),\ (0,2,1),\ (2,0,1),\ (2,1,0),\ (1,2,0),\ (1,0,2), & \\ 
  \nonumber & (0,-1,2),\ (0,2,-1),\ (2,0,-1),\ (2,-1,0),\ (-1,2,0),\ (-1,0,2), & \\ 
  \nonumber & (0,-1,-2),\ (0,-2,-1),\ (-2,0,-1),\ (-2,-1,0),\ (-1,-2,0),\ (-1,0,-2), & \\ 
  & (0,1,-2),\ (0,-2,1),\ (-2,0,1),\ (-2,1,0),\ (1,-2,0),\ (1,0,-2) \}, & 
\end{eqnarray} 
generated by the octahedral group ${\cal O}_3$ acting on the canonical vertex 
$O = (0,1,2)$. 
A useful rule of thumb is that $P = (x,y,z)$ represents a network vertex just 
when none of the 6 sums or differences of pairs of integer components vanishes:  
\begin{equation} \label{[thumb]} 
  x \pm y,\ y \pm z,\ z \pm x \ne 0 \pmod 4 . 
\end{equation} 
Each vertex $P$ has four network neighbours, at Euclidean separation $\sqrt{}2$; the neighbours of $O = (0,1,2)$ comprise 
\begin{eqnarray} \label{[nhbrs]} 
  \{ (0,2,1), (0,2,3), (1,0,2), (-1,0,2) \} . 
\end{eqnarray} 
\par

A {\sl sector} denotes any of 48 isomorphs of the canonical fundamental region 
\begin{eqnarray} 
  {\cal U} \equiv \{ (x, y, z)\ |\ x \ge y \ge z \ge 0 \} 
\end{eqnarray} 
under the action of ${\cal O}_3$. To any vertex $P = (x,y,z)$ corresponds an 
isomorph $P' = (x',y',z') \in {\cal U}$, where $x',y',z'$ denotes the same bag 
of integers as $x,y,z$, apart from dropping signs and sorting. 
Remark that $O = (0,1,2) \not= (0,0,0)$, the origin; also $O \notin {\cal U}$ ! 
\par 

% Sectors (qua point-sets) adjacent to axis line of T ... ?? 
With translation symmetry \\ 
${\rm T_3} : P \to P + (2,2,2)$ of path length 3, associate its {\sl quadrant} 
(union of 12 contiguous sectors) 
\begin{eqnarray} \label{[trans3]} 
  \{ (x, y, z)\ |\ x + y \ge 0 \ \&\  y + z \ge 0 \ \&\  z + x \ge 0 \} ; 
\end{eqnarray} 
${\rm T_4} : P \to P + (4,0,0)$ of length 4, its 8-sector quadrant (sextant?) 
\begin{eqnarray} \label{[trans4]} 
  \{ (x, y, z)\ |\ x \ge |y| \ \&\  x \ge |z| \} ; 
\end{eqnarray} 
${\rm T_6} : P \to P + (4,4,0)$ of length 6, its 4-sector quadrant 
\begin{eqnarray} \label{[trans6]} 
  \{ (x, y, z)\ |\ y \ge |z| \ \&\  x \ge y + |z| \} . 
\end{eqnarray} 
${\rm T_3}$ carries the canonical cell to one centred around a cubical centre; 
together with ${\cal O}_3$ it generates the network symmetry group. 
\par 

\section{Height Function} \label{[natdist]} 
\par 

Denote by $\dis(P)$ the {\sl height} of vertex $P$, the minimum path length 
from $O$ to $P$ along network edges. 
\par 

We introduce a mildly indigestible function $h(P)$ on vertices 
(revealed to equal height by Theorem \ref{[distthm]}), 
in terms of $h'(P')$ (equal to mean height over an ${\cal O}_3$ orbit): 
\begin{Def} \label{[disudef]} 
For $P' = (x',y',z') \in {\cal U}$, let 
\begin{eqnarray} \label{[distdef1]}
  h'(P') \equiv x' + y'/2 - s(P')/2 , 
\end{eqnarray} 
where 
\begin{eqnarray*} 
  s(P') &\equiv\quad (0,1,0,-1) &{\rm\ if\ } y' \bmod 4 = 0,1,2,3 , \\ 
      &\quad\times\ (1,0,-1,0) &{\rm\ if\ } x' \bmod 4 = 0,1,2,3 . 
\end{eqnarray*} 
\end{Def} \noindent 
Now extending this to all sectors, 
\begin{Def} \label{[distdef]} 
\begin{eqnarray} \label{[distdef2]}
  h(P) \equiv h'(P') + \begin{cases} 
    0 & {\rm if\ } |x| \ge |y| \ge |z| \\ 
    - \sig(z) & {\rm if\ } |x| \ge |z| \ge |y| \\ 
    - \sig(y) & {\rm if\ } |y| \ge |x| \ge |z| \\ 
    0 & {\rm if\ } |y| \ge |z| \ge |x| {\rm\ and\ } y z \le 0 \\ 
    -2 \sig(z) & {\rm if\ } |y| \ge |z| \ge |x| {\rm\ and\ } y z \ge 0 \\ 
    -2 \sig(z) & {\rm if\ } |z| \ge |x| \ge |y| \\ 
    - \sig(z) & {\rm if\ } |z| \ge |y| \ge |x| {\rm\ and\ } y z \le 0 \\ 
    -3 \sig(y) & {\rm if\ } |z| \ge |y| \ge |x| {\rm\ and\ } y z \ge 0 \\ 
  \end{cases} 
\end{eqnarray} 
\end{Def} \noindent 
The sector offsets involved above are illustrated in Figure \ref{[crossect]}, 
in the form of a section across the $x$ axis, $y$ running 
top to bottom, $z$ left to right, and origin (beneath) centre. 
Remark how sectors with a common coordinate plane share equal offsets; also 
$(x,y,z) \to (-x,y,z)$ is a symmetry, $(x,y,z) \to (x,z,y)$ an antisymmetry. 

\delete{ 
Cross-sections perp. to x-axis of fundamental regions: 
           ___________________
          |\   |    |    |   /|  
          | \+2|    |    | 0/ |  
          |+3\ | +1 | +1 | /-1|  
          |___\|____|____|/___|  
          |    |\   |   /|    |  
          | +2 | \ 0| 0/ | -2 |    \ z -> 
          |    |+1\ | /-1|    |    y 
          |____|___\|/___|____|    |   (o) x 
          |    |   /|\   |    |    v 
          |    |+1/ | \-1|    |  
          | +2 | / 0| 0\ | -2 |  
          |____|/___|___\|____|  
          |   /|    |    |\   |  
          |+1/ | -1 | -1 | \-3|  
          | / 0|    |    |-2\ |  
          |/___|____|____|___\|  
} 

\def \st {\strut} 
\begin{figure}[htb] 
\centering
\begin{tikzpicture}
\matrix [matrix of nodes,text height=8pt,text depth=2pt,text width=0.5cm,nodes={align=center}] (mat)
{ % no blank line before brace! 
 \st &\st &$+2$ &\st &\st &\st  &\st &\st &\st  &$ 0$&\st &\st \\
 \st &\st &\st  &\st &$+1$&\st  &\st &$+1$&\st  &\st &\st &\st \\
 $+3$&\st &\st  &\st &\st &\st  &\st &\st &\st  &\st &\st &$-1$\\

 \st &\st &\st  &\st &\st &$ 0$ &$ 0$&\st &\st  &\st &\st &\st \\
 \st &$+2$&\st  &\st &\st &\st  &\st &\st &\st  &\st &$-2$&\st \\
 \st &\st &\st  &$+1$&\st &\st  &\st &\st &$-1$ &\st &\st &\st \\

 \st &\st &\st  &$+1$&\st &\st  &\st &$\cdot\, O$ &$-1$ &\st &\st &\st \\
 \st &$+2$&\st  &\st &\st &\st  &${\cal U}$&\st &\st  &\st &$-2$ &\st \\
 \st &\st &\st  &\st &\st &$ 0$ &$ 0$&\st &\st  &\st &\st &\st \\

 $+1$&\st &\st  &\st &\st &\st  &\st &\st &\st  &\st &\st &$-3$\\
 \st &\st &\st  &\st &$-1$&\st  &\st &$-1$&\st  &\st &\st &\st \\
 \st &\st &$ 0$ &\st &\st &\st  &\st &\st &\st  &$-2$&\st &\st \\
}; 

% \draw (mat-1-1.north west) -- (mat-1-12.north east); 
 \draw (mat-4-1.north west) -- (mat-4-12.north east); 
 \draw (mat-7-1.north west) -- (mat-7-12.north east); 
 \draw (mat-10-1.north west) -- (mat-10-12.north east); 
% \draw (mat-12-1.south west) -- (mat-12-12.south east); 

% \draw (mat-1-1.north west) -- (mat-12-1.south west); 
 \draw (mat-1-4.north west) -- (mat-12-4.south west); 
 \draw (mat-1-7.north west) -- (mat-12-7.south west); 
 \draw (mat-1-10.north west) -- (mat-12-10.south west); 
% \draw (mat-1-12.north east) -- (mat-12-12.south east); 

 \draw (mat-1-1.north west) -- (mat-12-12.south east); 
 \draw (mat-1-12.north east) -- (mat-12-1.south west); 
 \draw [ultra thick] (mat-7-7.north west) -- (mat-9-9.south east); 
 \draw [ultra thick] (mat-9-7.south west) -- (mat-9-9.south east); 
 \draw [ultra thick] (mat-7-7.north west) -- (mat-9-7.south west); 

\end{tikzpicture}
\caption{Height offsets by sector: $x$ out, $y$ down, $z$ right, origin central.} 
\label{[crossect]} 
\end{figure}
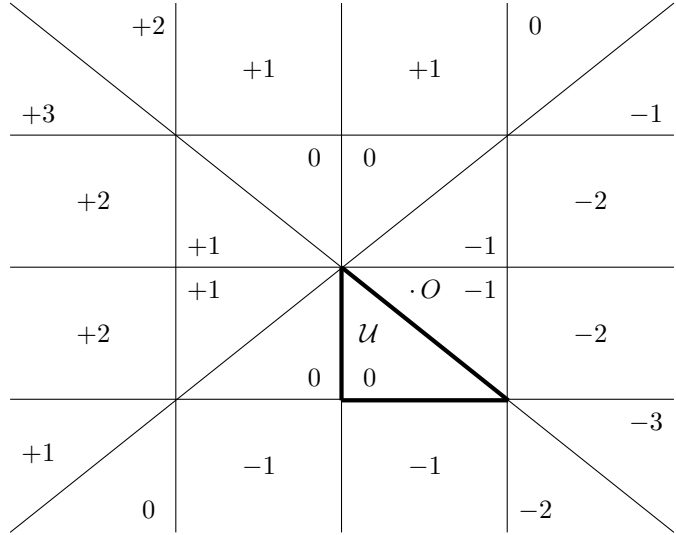 

Within its quadrant, a translation `respects' the pretender height function: 
\begin{Lem} \label{[transrxy]} 
\begin{eqnarray*} 
  h({\rm T_3}P) = h(P) + 3 &{\rm\ for\ }& P \in {\rm T_3} {\rm\ quadrant}; \\ 
  h({\rm T_4}P) = h(P) + 4 &{\rm\ for\ }& P \in {\rm T_4} {\rm\ quadrant}; \\ 
  h({\rm T_6}P) = h(P) + 6 &{\rm\ for\ }& P \in {\rm T_6} {\rm\ quadrant}. 
\end{eqnarray*} 
\end{Lem} 
\begin{proof} 
%For ${\rm T}$ a tiling symmetry, ${\rm T}P$ is a vertex just when $P$ is. 
Within the ${\rm T_3}$ quadrant, suppose without loss of generality that 
$|x| \ge |y| \ge |z|$; then $x \ge y \ge 0$, and via Definition \ref{[distdef]} 
\begin{eqnarray*} 
  h({\rm T_3}P) &=& (x+2) + (y+2)/2 + s(x+2, y+2, z+2)/2 \\ 
  &=& x + y/2 + s(x, y, z)/2 + 3 = h(P) + 3 . 
\end{eqnarray*} 
Similarly for ${\rm T_4}$, supposing that $|y| \ge |z|$, 
\begin{eqnarray*} 
  h({\rm T_4}P) &=& (x+4) + y/2 + s(x+4, y, z)/2 \\ 
  &=& x + y/2 + s(x, y, z)/2 + 4 = h(P) + 4 . 
\end{eqnarray*} 
For ${\rm T_6}$, 
\begin{eqnarray*} 
  h({\rm T_6}P) &=& (x+4) + (y+4)/2 + s(x+4, y+4, z)/2 \\ 
  &=& x + y/2 + s(x, y, z)/2 + 6 = h(P) + 6 . 
\end{eqnarray*} 
\end{proof} \noindent 
\par 

Given a vertex pretending to sufficient height, a useable neighbourhood of it 
lies entirely within some quadrant: 
\begin{Lem} \label{[transect]} 
If vertex $P \in {\cal U}$ with $h(P) \ge 15$, then $P$ together with its 
neighbours, and their images under the associated inverse translation 
${\rm T}^{-1}$, lie within the same quadrant. 
\end{Lem} 
\begin{proof} 
The components of any neighbour $Q$ vary from those of $P$ by $-1,0,+1$. 
% ${\cal T}_j(Q)$ worst case; no T6 reason ?? 
Referring to Equation \ref{[trans3]} etc. --- 
\par 
Case $y + z > 6$ : then $3 \le y \le x$. At worst, 
\begin{eqnarray*} 
  & P = (x, y, z),\ Q = (x, y-1, z-1), & \\ 
  & {\rm T_3}^{-1}P = (x-2, y-2, z-2),\ {\rm T_3}^{-1}Q = (x-2, y-3, z-3); & 
\end{eqnarray*} 
all remain within the quadrant, including ${\rm T_3}^{-1}Q$, since 
\begin{eqnarray*} 
x+y-5 \ge 0 \ \&\ y+z-6 \ge 0 \ \&\ z+x-5 \ge 0 . 
\end{eqnarray*} 
\par 
Case $y + z \le 6$ : then $0 \le y,z \le 6$, and $x \ge 12$ via Definition \ref{[distdef]}. 
At worst, 
\begin{eqnarray*} 
  & P = (x, y, z),\ Q = (x-1, y+1, z), & \\ 
  & {\rm T_4}^{-1}P = (x-4, y, z),\ {\rm T_4}^{-1}Q = (x-5, y+1, z); & 
\end{eqnarray*} 
all remain within the quadrant, including ${\rm T_4}^{-1}Q$, since 
\begin{eqnarray*} 
x-|y|-6 \ge 0 \ \&\ x-|z|-5 \ge 0 . 
\end{eqnarray*} 
\end{proof} \noindent 

\begin{Thm} \label{[distthm]} 
$\dis(P) = h(P)$ for every network vertex $P$. 
\end{Thm} 
\begin{proof} 
For $\dis(P) < 15$ the assertion is verified via inspection of an inconveniently 
extensive table. For $\dis(P) \ge 15$ via a somewhat delicate induction: 
assume the result for all vertices $R$ with $\dis(R) < \dis(P)$. 
\par 
First suppose $P \in {\cal U}$. 
$P$ has some neighbour $Q$ nearer to $O$, so that $\dis(Q) = \dis(P) - 1$; 
via Lemma \ref{[transect]} all $P,T^{-1}(P),Q,T^{-1}(Q)$ lie in one quadrant, 
and via Lemma \ref{[transrxy]} and network symmetry 
\begin{equation*} 
  \dis(P) = \dis(Q) + 1 = h(Q) + 1 = h({\rm T}^{-1}Q) + t+1 = h({\rm T}^{-1}P) + t = h(P) , 
\end{equation*} 
where $t = 3,4$ for ${\rm T} = {\rm T_3},{\rm T_4}$ resp. 
\par 
For any vertex $P$ now, the procedure above may be applied to its sector, 
employing an appropriate symmetry from ${\cal O}_3$, % ref [symmetry] ? 
and corresponding height offset from Definition \ref{[distdef]}. 
\end{proof} 
\par 

\section{Spherical Content Function} \label{[natcirc]} 
\par 

Consider now the content of a `sphere' in the corresponding metric, 
also the `vertex coordination sequence' of the tiling: 
that is, the number of vertices $S(n)$ at given height $n$ from $O$. 
\par 

We introduce functions (there are more to come!) for the content of a `sphere' 
of given height $n$, and of the `armillary sphere' where it meets the coordinate 
planes, and of their restrictions to the canonical sector: 
\begin{Def} \label{[bcdefn]} 
\begin{eqnarray*} 
  S(n) &\equiv& \#( P\ |\ \dis(P) = n ) ; \\ 
  {\bar S}(n) &\equiv& \#( P\ |\ \dis(P) = n\ \&\ x y z = 0 ) ; \\ 
  S_2(n) &\equiv& \#( P \in {\cal U}\ |\ \dis(P) = n ) ; \\ 
  {\bar S}_1(n) &\equiv& \#( P \in {\cal U}\ |\ \dis(P) = n\ \&\ x y z = 0 ) ; 
\end{eqnarray*} 
\end{Def} \noindent 

\begin{figure}[htb] 
\centering
\footnotesize 
\begin{tabular*}{0.99\textwidth}{@{\extracolsep{\fill}}|r|rrrrrrrrrrrrr|} 
\hline 
$   n =$& 0& 1&  2&  3&  4&  5&  6&  7&  8&  9& 10& 11& 12 \bigstrut \\ 
\hline 
${\bar S}(n) =$& 1& 4& 8& 12& 16& 20& 24& 28& 32& 36& 40& 44& 48 \bigstrut[t] \\ 
% [ 1, 4, 8, 12, 16, 20, 24, 28, 32, 36, 40, 44, 48, 52, 56, 60 ] 
$S(n) =$& 4& 10& 20& 34& 52& 74& 100& 130& 164& 202& 244& 290& 340 \bigstrut[b] \\ 
% [ 1, 4, 10, 20, 34, 52, 74, 100, 130, 164, 202, 244, 290, 340, 394, 452, 514 ] 
\hline 
\end{tabular*} 
\caption{Table of network sphere content} % Dummy for label  --- BUG ?? 
\label{[ctertab]} 
\end{figure} 

\begin{Lem} \label{[linecomb]} 
There are finitely many constants $a_j,b_j,c_j$ such that for all $n$, 
\begin{eqnarray*} 
{\bar S}(n) &=& \sum_j a_j {\bar S}_1(n + j) , \\ 
S(n) &=& \sum_j b_j {\bar S}_1(n + j) + \sum_j c_j S_2(n + j) . 
\end{eqnarray*} 
\end{Lem} 
\begin{proof} 
Firstly notice that via Equation \ref{[thumb]} the only tiling vertices on the 
boundary of a sector are those interior to a facet on a coordinate plane, 
since the other two planes are diagonal. 
%--- somewhat simplifying the subsequent computation. 
\par 
Now via Definition \ref{[distdef2]} and Theorem \ref{[distthm]}, for $n > 0$, 
\begin{eqnarray*} 
  {\bar S}(n) &=& 6 {\bar S}_1(n) + 3 {\bar S}_1(n-1) + 3 {\bar S}_1(n+1) + {\bar S}_1(n+2) + {\bar S}_1(n-2) \\ 
       && +\ {\bar S}_1(n+3) + {\bar S}_1(n-3) + 2 {\bar S}_1(n+1) + 2 {\bar S}_1(n-1) \\ 
       && +\ 2 {\bar S}_1(n+2) + 2 {\bar S}_1(n-2) ; \\ 
  S(n) &=& 2\bigl( 6 S_2(n) + 3 S_2(n-1) + 3 S_2(n+1) + S_2(n+2) + S_2(n-2) \\ 
       && +\ S_2(n+3) + S_2(n-3) + 2 S_2(n+1) + 2 S_2(n-1) \\ 
       && +\ 2 S_2(n+2) + 2 S_2(n-2) \bigr) - {\bar S}(n) ; \\ 
\end{eqnarray*} 
the final term corrects for the boundary being counted double, and is then 
substituted via the first equation. 
\end{proof} \noindent 
%Although for concreteness coefficients $a_j,b_j,c_j$ are written out explicitly 
%above, these details are in practice unimportant: only their span is significant. 
\par 

\begin{Thm} \label{[surfun]} 
\begin{equation*} 
  {\bar S}(n) = \begin{cases} 
    1 & {\rm if\ } n = 0 , \\ 
    4 n & {\rm\ if\ } n > 0 . 
  \end{cases} 
\end{equation*} 
\begin{equation*} 
  S(n) = \begin{cases} 
    1 & {\rm if\ } n = 0 , \\ 
    2(n^2 + 1) & {\rm\ if\ } n > 0 ; 
  \end{cases} 
\end{equation*} 
\end{Thm} 
\begin{proof} 
$P = (x,y,z)$ 
will be restricted implicitly to vertices at height $n$ in sector ${\cal U}$. 
Also assume $n > 6$, avoiding special values for $n \le 0$. 
Lemma \ref{[transrxy]} is employed without reference, noting ${\cal U}$ is a subset of all relevant quadrants. 
A leaning tower of further subsidiary functions follows: 
\par 
Firstly via $P \to {\rm T_4}^{-1}P$ from Equation \ref{[trans3]} etc., 
\begin{eqnarray*} 
  {\bar S}_1(n) &\equiv& \#( P\ |\ 0 = z \le y \le x )  =  {\bar S}_1(n-4) + {\bar S}_0(n) ; 
\end{eqnarray*} 
where via $P \to {\rm T_6}^{-1}P$, 
\begin{eqnarray*} 
  {\bar S}_0(n) &\equiv& \#( P\ |\ x-4 < y \le x )  =  {\bar S}_0(n-6) = {\rm constant}, 
\end{eqnarray*} 
depending on $n \bmod 6$, $n \bmod 4$. 
Hence the ${\bar S}_1(12 i + j)$ are polynomials linear in $i$, depending only 
on $j = n \bmod 12$. 
\par 
Similarly via $P \to {\rm T_3}^{-1}P$, 
\begin{eqnarray*} 
  S_2(n) &\equiv& \#( P\ |\ 0 \le z \le y \le x )  =  S_2(n-3) + S_1(n) ; 
\end{eqnarray*} 
where via $P \to {\rm T_6}^{-1}P$, 
\begin{eqnarray*} 
  S_1(n) &\equiv& \#( P\ |\ 0 \le z < 2 )  =  S_1(n-6) + S_0(n) ; 
\end{eqnarray*} 
where via $P \to {\rm T_4}^{-1}P$, 
\begin{eqnarray*} 
  S_0(n) &\equiv& \#( P\ |\ 0 \le z < 2\ \&\ z \le y < 4 ) 
  =  S_0(n-4) = {\rm constant}, 
\end{eqnarray*} 
depending on $n \bmod 3$, $n \bmod 6$, $n \bmod 4$. 
Hence the $S_2(12 i + j)$ are polynomials quadratic in $i$, depending only on 
$j = n \bmod 12$. 
\par 
%As before, further explicit details of the numerous polynomials involved above 
%may be ignored. 
For via Lemma \ref{[linecomb]}, both ${\bar S}(12 i + j), S(12 i + j)$ are 
also such sets of polynomials; and by inspection of Figure \ref{[ctertab]}, 
each set reduces (mysteriously) to the single polynomial in $n > 0$ shown. 
\end{proof} 
\par 

\section{Discussion} \label{[discus]} 
\par 

Coordination sequences for networks associated with lattices in Euclidean 
{$d$-space} are established in \cite{[Con97]} and \cite{[Bac98]} for numerous 
classical cases, using general group-theoretic methods. However, an explicit 
expression for the coordination sequence $S(n)$ in Theorem \ref{[surfun]} 
--- case $d = 3$ of Equation (3.43) in \cite{[Con97]} --- 
was designated conjectural, 
and has since been dubbed by one author `really tricky'. 	 % BOAST ?? 
\par 

The approach used above relies on two ideas. One is intuitively obvious: 
for all vertices $P$ lying sufficiently far from the height zero vertex $O$ 
in the direction of translation ${\rm T}$, translation must respect height, 
in the sense that $\dis({\rm T}\,P) = \dis(P) + t$ for constant $t$. 
As a result, the sphere of height $n + t$ is the union of translations of 
overlapping segments of the sphere of height $n$. The difficulty comes in 
quantifying regions in which this respectful behaviour can be guaranteed, 
for some neighbourhood of $P$ sufficient to facilitate induction: 
it is overcome by overlapping adjacent quadrants so far, that $P$ cannot 
avoid lying well inside (at least) one. 
\par 

Given the explicit expression Definition \ref{[distdef]} for height, 
it would be feasible (though tedious) to compute content $S(n)$ as the number of 
solutions of the compound linear equation $h(x,y,z) = n$. However, it is not 
actually necessary to know $\dis(P)$ in order to find $S(n)$; 
only Lemma \ref{[transrxy]}, that outside some finite initial region, 
translation respects height. 
\par 

References to `inspection' are in practice made tongue-in-cheek. Frequent 
resort to a computer is inevitable if the proof is to be fully checked, 
starting with tabulation of $\dis(P) \le 20$ via some tree-search algorithm, 
and implementation of $h(P)$, and extending either to compute $S(n)$ by 
enumeration. 
\par 

Various features of the proofs merit further discussion. \\ 
\noindent 
(A) The coordinates employed above for sodalite are special to dimension 
$d = 3$, and do not generalise to other spaces. Conventional coordinates 
for {d-space} sodalite (aka `honeycomb of permutohedra') are sketched in 
\cite{[WikPH]}. 
\par 

\noindent 
(B) The proof of Lemma \ref{[linecomb]} and Theorem \ref{[surfun]} involved 
establishing existence of moderately complicated linear combinations and 
polynomial sets; however, their actual coefficients are ultimately irrelevant. 
\par 

\noindent 
(C) Indeed, if only we could be certain in advance that the the final result 
in Theorem \ref{[surfun]} would be a quadratic polynomial (or more generally, 
ultimately satisfied some linear recurrence with constant coefficients and 
known order), only minor routine computation would subsequently be required. 
\par 

\noindent 
(D) Although in this instance height $\dis(P)$ is represented by expression 
$h(P)$ for all $\dis(P) \ge 0$, content $S(n)$ takes special values unless 
$n = \dis(P) \ge 1$. In other situations it may be nontrivial to establish a 
lower bound on height above which the corresponding functions behave well. 
\par 

\noindent 
(E) Larger still is the bound --- $\dis(P) \ge 15$ in Lemma \ref{[transect]} 
--- required to ensure that height is respected by inverse translation, 
below which we are obliged to verify a result via inspection. 
Unfortunately this bound depends in a complicated fashion on the geometric 
interaction between translations, and may well grow rapidly with dimension. 
\par 

\noindent 
(F) The quadrant associated with a translation ${\rm T}$ would in general 
be the union of those sectors having an edge along its axis. This rule is 
subverted by ${\rm T}_3$ in Equation \ref{[trans3]}, which spills over into 
a further 6 adjacent sectors. The effect is to simplify the proof of 
Lemma \ref{[transect]}, which otherwise would otherwise involve (dimension) 
$d = 3$ translations including ${\rm T}_6$. 
\par 

\noindent 
(G) The proof of Lemma \ref{[linecomb]} would be considerably complicated by 
any necessity to consider vertices lying on lower-dimensional boundary elements 
(here edges, corner) of a sector. 
\par 

\noindent 
(H) Assignment of coordinates labelling vertices is a delicate matter, 
governed by availability of convenient generators for the symmetry group. 
In particular, it may be inadvisable to locate vertex $O$ at origin $(0,0,0)$, 
or indeed to assign a vertex to the origin at all: 
instead the frame should reflect the full point-group of the network. 
\par 

All in all there appear to remain considerable obstructions to recasting 
our approach as an abstract theorem applicable to a class of networks, 
analogous to Theorems 2.4 and 2.9 of \cite{[Con97]}. 
In the meantime, the method has been successfully applied to the snub-square 
and knight's move networks. % references ?? poster? 
\par 

%\bibliographystyle{amsplain} 
%s\bibliography{knight_move} % insert .bib filename less extension! 
% Or later incorporate sodalite.bbl here ... 

%\nocite{*} % includes whole biblio 

%\delete{ 
\providecommand{\bysame}{\leavevmode\hbox to3em{\hrulefill}\thinspace}
\providecommand{\MR}{\relax\ifhmode\unskip\space\fi MR }
\providecommand{\MRhref}[2]{%
  \href{http://www.ams.org/mathscinet-getitem?mr=#1}{#2}
}
\providecommand{\href}[2]{#2}

%} 

\clearpage
%\closegraphsfile

\end{document}